\documentclass[11pt]{amsart}
\usepackage{float}
\usepackage{graphicx,youngtab}
\usepackage{amssymb}
\theoremstyle{plain}
\newtheorem{thm}{Theorem}[section]
\newtheorem{lemma}[thm]{Lemma}

\newtheorem{conj}[thm]{Conjecture}
\newtheorem{prop}[thm]{Proposition}

\theoremstyle{definition}

\newtheorem{rmk}[thm]{Remark}
\newtheorem{example}[thm]{Example}

\newcommand{\bbC}{\mathbb{C}}

\newcommand{\bbR}{\mathbb{R}}

\newcommand{\bbZ}{\mathbb{Z}}

\newcommand{\caM}{\mathcal{M}}
\newcommand{\caN}{\mathcal{N}}

\begin{document}

\title[Antichain generating polynomials of posets]
{Antichain generating polynomials of posets}

\author{Jian Ding}
\address[Ding]{College of Mathematics and Econometrics, Hunan University, Changsha 410082,
China}
\email{dingjain@hnu.edu.cn}

\author{Chao-Ping Dong}
\address[Dong]{Mathematics and Science College, Shanghai Normal University, Shanghai 200234,
P.~R.~China}
\email{chaopindong@163.com}
\thanks{Dong is supported by NSFC grant 11571097 and Shanghai Gaofeng Project for University Academic Development Program.}

\abstract{}
This paper gives a formula for the antichain generating polynomial $\caN_{[k]\times Q}$ of the poset $[k]\times Q$, where $[k]$ is an arbitrary chain and $Q$ is any finite graded poset.
When $Q$ specializes to be a connected minuscule poet, which was classified by Proctor in 1984, we find that the polynomial $\caN_{[k]\times Q}$ bears nice properties. For instance, we will recover the $B_n$-Narayana polynomial and the $D_{2n+2}$-Narayana polynomial. We collect evidence for the conjecture that whenever $\caN_{[k]\times P}(x)$ is palindromic, it must be $\gamma$-positive. Moreover, the family $\caN_{[2]\times [n]\times [m]}$ should be real-rooted and $\caN_{[2]\times [n]\times [n+1]}$ should be $\gamma$-positive.  We also conjecture that $\caN_{Q}(x)$ is log-concave (thus unimodal) for any connected Peck poset $Q$.
\endabstract



\maketitle

%\tableofcontents

\section{Introduction}

As on page 244 of Stanley \cite{St}, we call a finite poset $Q$ \emph{graded} if every maximal chain in $Q$ has the same length. In this case, there is a unique rank function $r: Q\to \bbZ_{> 0}$ such that all the minimal elements have rank $1$, and $r(x)=r(y)+1$ if $x$ covers $y$. Let $Q_i$ denote the set of all the elements in $Q$ having rank $i$. We call $Q_i$ a \emph{rank level} of $Q$. Let $d$ be the maximum of the rank function $r$. Then we partition $Q=\bigsqcup_{i=1}^d Q_i$, $1\leq i\leq d$, into $d$ rank levels. If
$|Q|_i=|Q|_{d+1-i}$ for $1\leq i\leq \frac{d}{2}$, we say that $Q$ is \emph{rank symmetric}. If $|Q|_1\leq |Q_2|\leq \cdots \leq |Q_k| \geq |Q_{k+1}|\geq \cdots \geq |Q_{d}|$ for some $1\leq k\leq d$, we say that $Q$ is \emph{rank unimodal}.

From now on, every poset $Q$ is assumed to be finite, graded, and connected. A subset $I$ of $Q$ is called an \emph{ideal} if $x\leq y$ in $Q$ and $y\in I$ implies that $x\in I$.  Put
$$
\caM_Q(x):=\sum_{I} x^{|I|},
$$ where $I$ runs over the ideals of $Q$.  A subset $A$ of $Q$ is called an \emph{antichain} if its elements are mutually incomparable.  Put
$$
\caN_Q(x):=\sum_{A} x^{|A|}
,$$
where $A$ runs over the antichains of $Q$.
Since ideals of $Q$ are in bijection with antichains of $Q$ via the map $I\to \max (I)$, where $\max (I)$ denotes the maximal elements of $I$, we have that
\begin{equation}\label{M-N-1}
\caM_Q(1)=\caN_Q(1).
\end{equation}
The ideal generating polynomial $\caM_Q(x)$ has been addressed intensively in the literature, while the antichain generating polynomial $\caN_Q(x)$ seems to attract much fewer attention. One possible reason for this  is that the computation of $\caN_Q(x)$ is much harder. The first result of the current paper is a formula for calculating the $\caN$-polynomial of $[k]\times Q$. Indeed,  let us use $I$, $I_j$, $J$, etc to denote ideals of $Q$. It is well-known that ideals of $[k]\times Q$ are in bijection with increasing $k$-sequences of ideals of $Q$: $I_1\subseteq I_2 \subseteq \cdots \subseteq I_k$. Then one sees that the corresponding antichain has size
\begin{equation}\label{antichain-size}
\sum_{j=1}^{k} \#\big(\max(I_j)\setminus I_{j-1}\big).
\end{equation}
Here we make the convention that $I_0$ is the empty ideal. Let us define
\begin{equation}\label{N-I-k}
\caN_{I}^k(x):=\sum_{I_1\subseteq  \cdots \subseteq I_{k-1}\subseteq I} x^{\sum_{j=1}^k\#\big(\max(I_j)\setminus I_{j-1}\big)},
\end{equation}
where $I_1\subseteq  \cdots \subseteq I_{k-1}$ runs over the increasing $(k-1)$-sequences of ideals of $Q$ which are contained in $I$.
Then
\begin{equation}\label{N-I-k-sum}
\caN_{[k]\times Q}(x)=\sum_{I}\caN_{I}^k(x),
\end{equation}
where $I$ runs over ideals of $Q$. Inspecting the formula \eqref{N-I-k} gives that
\begin{equation}\label{N-I-k-K+1}
\caN_{I}^{k+1}(x):=\sum_{J\subseteq I} x^{\#\big(\max(I)\setminus J\big)} \caN_J^k(x),
\end{equation}
where $J\subseteq I$ runs over ideals of $Q$. This leads us to the following.

\medskip
\noindent\textbf{Theorem A.}
\emph{Let $Q$ be a finite, connected, and graded poset. Let $I_1$, $I_2$, $\dots$, $I_{N(Q)}$ be a enumeration of all the ideals of $Q$.  Let $A_Q$ be the $N(Q)\times N(Q)$ matrix whose $(i, j)$-entry $a_{ij}$ equals $x^{\#\big(\max(I_i)\setminus I_j\big)}$ if $I_j\subseteq I_i$; and $a_{ij}$ equals zero otherwise. Let $V_Q^{(k)}$ be the column vector whose $i$-th entry is $\caN_{I_i}^k(x)$ for $1\leq i\leq N(Q)$. Then
\begin{equation}\label{N-I-k-K+1-matrix}
V_Q^{(k)}=A_Q^{k-1}V^{(1)}_Q.
\end{equation}}
\medskip

The formula \eqref{N-I-k-K+1-matrix} turns the calculation of $V_Q^{(k)}$---hence the calculation of $\caN_{[k]\times Q}(x)$ in view of \eqref{N-I-k-sum}---into matrix multiplication. See Example \ref{exam-2-3} for an illustration. Computationally, this is very efficient.

For the rest of the paper, let us specialize $Q$ to be a minuscule poset, and demonstrate that  similar to the polynomial $\caM_{[k]\times Q}$, the $\caN$-polynomial could also bear very nice properties. However, the techniques for unveiling them should lie much deeper.

Minuscule posets arise from minuscule representations of simple Lie algebras over $\bbC$.
According to Section 11 of Proctor \cite{Pr84}, minuscule posets are ubiquitous in mathematics. Our limited understanding of this philosophy comes from the recent work \cite{DW}, where  certain root posets arising from $\bbZ$-gradings of simple Lie algebras turn out to be mainly decoded by $[k]\times P$ for minuscule posets $P$.

As been classified by Proctor \cite{Pr84}, a connected minuscule poset $P$ is one of the following:
\begin{itemize}
\item[$\bullet$] $[m]\times [n]$;
\item[$\bullet$] $H_n:=([n]\times [n])/S_2$;
\item[$\bullet$] $K_n=[n]\oplus([1]\sqcup [1])\oplus [n]$ (the ordinal sum, see page 246 of Stanley \cite{St});
\item[$\bullet$]    $J^2([2]\times [3])$ (see Fig.~\ref{fig-J2});
\item[$\bullet$]    $J^3([2]\times [3])$.
\end{itemize}
Here $J(P)$ is the poset consisting of the ideals of $P$, partially ordered by inclusion; $J^2(P)$ stands for $J(J(P))$ and so on. The $\caM$-polynomials of minuscule posets enjoy many nice properties. For instance, by Theorem 6 of Proctor \cite{Pr84},
\begin{equation}\label{M-k-P}
\caM_{[k]\times P}(x)=\prod_{\alpha\in P}\frac{1-x^{r(\alpha)+k}}{1-x^{r(\alpha)}}.
\end{equation}
Due to the symmetry of $[k]\times P$, the polynomial $\caM_{[k]\times P}(x)$ is palindromic.

Let $f(x)=\sum_{i=0}^{n}a_i x^i\in\bbR_{>0}[x]$ be a polynomial of degree $n$. We say that $f(x)$ is \emph{palindromic} if $a_i=a_{n-i}$ for $0\leq i\leq n$; that $f(x)$ is \emph{monic} if $a_n=1$; and that $f(x)$ is \emph{$\gamma$-positive} if there exist some positive reals $\gamma_0, \gamma_1, \dots, \gamma_{\lfloor \frac{n}{2}\rfloor}$ such that
\begin{equation}\label{gamma-pos}
f(x)=\sum_{i=0}^{\lfloor \frac{n}{2}\rfloor} \gamma_i x^i (1+x)^{n-2i}.
\end{equation}
In this case, we call $\gamma_0$, $\dots$, $\gamma_{\lfloor \frac{n}{2}\rfloor}$ the $\gamma$-coefficients of $f$.
We say that $f(x)$ is \emph{real-rooted} if every root of $f(x)$ is real; that $f(x)$ is \emph{log-concave} if $a_i^2\geq a_{i-1}a_{i+1}$ for $1\leq i\leq n-1$. See Athanasiadis \cite{At} and Stanley \cite{St89}  for excellent surveys on $\gamma$-positivity, log-concavity and unimodality in algebra, combinatorics and geometry.

We shall collect evidence for the following two conjectures. Theorem A is very helpful in this process.

\medskip
\noindent\textbf{Conjecture B.}
\emph{Let $P$ be a connected minuscule poset. Let $k$ be any positive integer. If $\caN_{[k]\times P}(x)$ is palindromic, then it must be $\gamma$-positive.}
\medskip

 We also suspect that the family
$\caN_{[2]\times [n]\times [m]}$ should be real-rooted and $\caN_{[2]\times [n]\times [n+1]}$ should be $\gamma$-positive, see Conjectures \ref{conj-real-rooted-2-n-k} and \ref{conj-2-n-n+1}. What underlies the hypothetical real-rootedness should  be abundance of interlacing relations, see Conjecture \ref{conj-inter-mat}.

One may view $\gamma$-positivity as a delicate property which happens only under very special circumstances. On the other hand, log-concavity lives in a much broader domain. For the latter, we propose the following.

\medskip
\noindent\textbf{Conjecture C.}
\emph{Let $Q$ be a connected finite Peck poset. The polynomial $\caN_{Q}(x)$ is log-concave (hence unimodal).}
\medskip

Recall that $Q$ is said to be \emph{Sperner} if no antichain has more elements than the largest rank level of $Q$ does.
We say that $Q$ is \emph{strongly Sperner} if for every $k\geq 1$ no union of $k$ antichains contains  more elements than the union of the $k$ largest rank levels of $Q$ does. We call $Q$ \emph{Peck} if it is strongly Sperner, rank symmetric and rank unimodal.

The paper is organized as follows. We collect necessary preliminaries in Section 2.  Section 3 aims to collect evidence for Conjectures B and C. Sections 4 focuses on the family $\caN_{[2]\times [n]\times [n+1]}(x)$.

\section{Preliminaries}
This section aims to give some preliminaries.
\begin{lemma}\label{lemma-gamma-positive}
Let $f(x)$ be a polynomial in $\bbR_{> 0}[x]$.
If $f(x)$ is real-rooted and palindromic, then it is $\gamma$-positive.
\end{lemma}
\begin{proof}
See Lemma 4.1 of Br\"and\'en \cite{Br}, Remark 3.1.1 of Gal \cite{Ga}, and Sun-Wang-Zhang \cite{SWZ}.
\end{proof}

\begin{lemma}\label{lemma-unique-rank-level}
Let $Q=\bigsqcup_{i=1}^d Q_i$ be the decomposition of a connected finite graded poset into rank levels. Assume that $Q$ is rank unimodal. Then $[k]\times Q$ has a unique rank level of the largest size if and only if $k\leq d$ and that there exists  a unique $i_0\in [k, d]$ such that the statistic $|Q_i|+|Q_{i-1}|+\cdots+|Q_{i+1-k}|$ attains the maximum at $i_0$.
\end{lemma}
\begin{proof}
When $k\geq d+1$, one sees easily that $[k]\times Q$ has $k-d+1$ rank levels of size $|Q|$, which must be the largest. Thus it remains to consider $k\leq d$, then observe that the rank levels of $[k]\times Q$ have sizes $|Q_i|+|Q_{i-1}|+\cdots+|Q_{i+1-k}|$ for $i\in [1, d+k-1]$. Here we interpret $Q_j$ as the empty set if $j$ does not fall in $[1, d]$.
Since $Q$ is assumed to be rank unimodal, we see that a necessary condition for the size to be largest is that $i\leq d$ and that $i+1-k\geq 1$, i.e., $i\in [k, d]$. Now the desired conclusion is obvious.
\end{proof}

\begin{lemma}\label{lemma-monic}
Let $P$ be a connected minuscule poset.
The polynomial $\caN_{[k]\times P}$ is monic precisely in the following cases:
\begin{itemize}
\item[(a)] $P=[m]\times [n]$, $m\leq n$, and $k=n-m+1, n-m+3, n-m+5,  \dots, n+m-1$;
\item[(b)] $P=H_n$, $n$ is odd, and $k=1, 5, 9, \dots, 2n-1$;
\item[(c)] $P=H_n$, $n$ is even, and $k=3, 7, 11, \dots, 2n-1$;
\item[(d)] $P=K_n$, and $k=1, 2n+1$;
\item[(e)] $P=J^2([2]\times [3])$, and $k=5, 11$;
\item[(f)] $P=J^3([2]\times [3])$, and $k=1, 9, 17$.
\end{itemize}
\end{lemma}
\begin{proof}
Since each connected minuscule poset is Peck and the product of Peck posets is Peck (see Theorem 2 of Proctor \cite{Pr82}), we have that $[k]\times P$ is Peck. Therefore, $[k]\times P$ is Sperner. Thus the polynomial $\caN_{[k]\times P}$ is monic if and only if $[k]\times P$ has a unique rank level of the largest size. Checking the latter condition via Lemma \ref{lemma-unique-rank-level} leads us to the desired conclusion.
\end{proof}

\begin{prop}\label{prop-gamma-n-n} The polynomials $\caN_{[n]\times [n]}(x)=\sum_{i\geq 0}{n \choose i}{n\choose
i}x^i$ are $\gamma$-positive.
\end{prop}
\begin{proof}
By (62) of \cite{At}, the polynomial $\caN_{[n]\times [n]}(x)$ coincides with ${\rm Cat}(B_n, x)$---the $B_n$-Narayana polynomial. Now by Proposition 11.15 of \cite{PRW} Postnikov, Reiner and Williams (see also Theorem 2.32 of Athanasiadis \cite{At}), these polynomials are $\gamma$-positive. Indeed, we have that
$$
{\rm Cat}(B_n, x)=\sum_{i=0}^{\lfloor \frac{n}{2}\rfloor} {n\choose i, i, n-2i} x^{i}(1+x)^{n-2i}.
$$
\end{proof}

\begin{prop}\label{prop-gamma-H-n} The polynomials $\caN_{H_n}(x)=\sum_{i\geq 0}{n+1 \choose 2i} x^i$ are $\gamma$-positive for $n$ odd.
\end{prop}
\begin{proof}
Let ${\rm E}_2$ be the linear operator on the space $\bbR[x]$ of polynomials with real coefficients defined by setting ${\rm E}_2(x^m)=x^{m/2}$ if $m$ is even, and ${\rm E}_2(x^m)=0$ otherwise. Then one sees easily that
$$
{\rm E}_2(1+x)^{n+1}=\sum_{i\geq 0}{n+1 \choose 2i} x^i.
$$
Since $(1+x)^{n+1}$ is real-rooted, we conclude from Lemma 7.4 of Athanasiadis and Savvidou \cite{AS} that ${\rm E}_2(1+x)^{n+1}$ is real-rooted.  Now Lemma \ref{lemma-gamma-positive} finishes the proof.
\end{proof}
\begin{rmk}
As suggested by Athanasiadis, it would be interesting to find combinatorial interpretations of the $\gamma$-coefficients of $\caN_{H_n}(x)$ when $n$ is odd.
\end{rmk}

Given two real-rooted polynomials $f(x)=\prod_{i=1}^{{\rm deg} f}(x-x_i)$ and $g(x)=\prod_{j=1}^{{\rm deg} g}(x-\xi_j)$, we say that $f(x)$ \emph{interlaces} $g(x)$---denoted by $f(x)\preceq g(x)$---if their roots alternate in the following way:
$$
\cdots\leq x_2\leq \xi_2\leq x_1\leq \xi_1.
$$
Note that a necessary condition for $f(x)\preceq g(x)$ is that ${\rm deg} f\leq {\rm deg} g\leq {\rm deg} f+1$.

\begin{thm}\emph{(\textbf{Obreschkoff} \cite[Satz 5.2]{Ob})}\label{thm-Ob}
 Let $f(x), g(x)\in\bbR[x]$ be two polynomials such that ${\rm deg} f\leq {\rm deg} g\leq {\rm deg} f+1$. Then $f(x)$ interlaces $g(x)$ if and only if $c_1 f(x)+c_2 g(x)$ is real-rooted for any $c_1, c_2\in\bbR$.
\end{thm}

 Given  three real-rooted polynomials $f(x)$, $g(x)$ and $h(x)$, after Chudnovsky and Seymour \cite{CS}, we  say that $h(x)$ is a \emph{common interleaver} for $f(x)$ and $g(x)$ if $f(x)\preceq h(x)$ and $g(x)\preceq h(x)$. Note that if $f(x)\preceq g(x)$, then $f(x)$ and $g(x)$ have a common interleaver $g(x)$.

\section{Evidence for Conjectures B and C}
This section aims to collect evidence for Conjectures B and C.

Firstly, let us verify Conjecture B for $k=1$.
\begin{prop}\label{prop-minuscule-gamma} Let $P$ be a connected minuscule poset. The following are equivalent:
\begin{itemize}
\item[(a)] $\caN_{P}(x)$ is monic;
\item[(b)] $\caN_{P}(x)$ is palindromic;
\item[(c)] $\caN_{P}(x)$ is $\gamma$-positive.
\end{itemize}
\end{prop}
\begin{proof}
It suffices to show that (a) implies (c). By Lemma \ref{lemma-monic}, $\caN_P(x)$ is monic precisely when $P$ is $[n]\times [n]$, or $H_n$ for $n$ odd, or $K_n$, or $J^3([2]\times [3])$. Moreover, we have that
\begin{itemize}
\item[$\bullet$] $\caN_{[n]\times [n]}(x)=\sum_{i\geq 0}{n \choose i}{n\choose
i}x^i$.

\item[$\bullet$] $\caN_{H_n}(x)=\sum_{i\geq 0}{n+1 \choose 2i} x^i$ for $n$ odd.

\item[$\bullet$] $\caN_{K_n}(x)=1+ (2n+2)x +x^2=(1+x)^2+2nx$.

\item[$\bullet$] $\caN_{J^3([2]\times [3])}=1+27x+27x^2+x^3=(1+x)^3+24x(1+x)$.
\end{itemize}
Now it follow from Propositions \ref{prop-gamma-n-n} and \ref{prop-gamma-H-n} that these polynomials are all $\gamma$-positive.
\end{proof}

Secondly,  let us verify Conjecture B for $P=K_n$.

\begin{prop}\label{prop-minuscule-gamma}  The following are equivalent:
\begin{itemize}
\item[(a)] $\caN_{[k]\times K_n}(x)$ is monic;
\item[(b)] $\caN_{[k]\times K_n}(x)$ is palindromic;
\item[(c)] $\caN_{[k]\times K_n}(x)$ is $\gamma$-positive.
\end{itemize}
\end{prop}
\begin{proof}
It suffices to show that (a) implies (c). By Lemma \ref{lemma-monic}, $\caN_{[k]\times K_n}(x)$ is monic precisely when $k=1$ or $2n+1$. It remains to consider $k=2n+1$.  Indeed, by Theorem 7.2 of \cite{DW},
\begin{eqnarray*}
\caN_{[2n+1]\times K_n}(x)&=&(1+x^{2n+2})+(4n^2+6n+2)(x+x^{2n+1})\\
    &+&\sum_{i=2}^{2n}\left({2n \choose i-2}{2n+1\choose i-1}+{2n\choose i}{2n+1\choose i}+2{2n+1 \choose i-1}{2n+1\choose i}\right)x^i.
\end{eqnarray*}
Notice that
\begin{eqnarray*}
& &{2n\choose i-2}{2n+1\choose i-1}+{2n\choose i}{2n+1\choose i}+2{2n+1\choose i-1}{2n+1\choose i}\\
&=&{2n+1\choose i-1}\bigg({2n+1\choose i}+{2n\choose i-2}\bigg)+{2n+1\choose i}\bigg({2n\choose i}+{2n+1\choose i-1}\bigg)\\
&=&{2n+1\choose i-1}\bigg({2n+1\choose i}+{2n\choose i-2}\bigg)+{2n+1\choose i}\bigg({2n\choose i}+{2n\choose i-1}+{2n\choose i-2}\bigg)\\
&=&{2n+1\choose i-1}\bigg({2n+1\choose i}+{2n\choose i-2}\bigg)+{2n+1\choose i}\bigg({2n+1\choose i}+{2n\choose i-2}\bigg)\\
&=&{2n+2\choose i}\bigg({2n+1\choose i}+{2n\choose i-2}\bigg).
\end{eqnarray*}
Therefore, it follows from (62) of \cite{At} that
$\caN_{[2n+1]\times K_n}(x)$ coincides with ${\rm Cat}(D_{2n+2}, x)$---the $D_{2n+2}$-Narayana polynomial.
The latter polynomial is $\gamma$-positive for any $n\in\bbZ_{>0}$ by Gorsky \cite{Go} (see also Theorem 2.32 of \cite{At}). Indeed, we have that
$$
{\rm Cat}(D_{2n+2}, x)=\sum_{i=0}^{n+1}\frac{2n+1-i}{2n+1}{2n+2\choose i, i, 2n+2-2i} x^i(1+x)^{2n+2-2i}.
$$
\end{proof}
\begin{rmk} It would be interesting to find direct explanations for
$$
\caN_{[n]\times [n]}(x)={\rm Cat}(B_n, x), \quad \caN_{[2n+1]\times K_n}(x)={\rm Cat}(D_{2n+2}, x).
$$
\end{rmk}

Now let us verify Conjecture B for the two exceptional minuscule posets.

\begin{example}\label{exam-J2}
We can parameterize the ideals of $Q:=J^2([2]\times [3])$ by those $4$-tuples $(a, b, c, d)$ such that $0\leq a\leq 4$, $b=0$ or $3\leq b\leq 6$, $c=0$ or $3\leq c\leq 6$, $d=0$ or $5\leq d\leq 8$ and that
\begin{equation}\label{ideal-J2}
a\geq \min\{b, 4\}, \quad b\geq c, \quad c\geq \min\{6, d\}, \quad d\geq 0.
\end{equation}
Let $I$ be the ideal of $Q$ which is parameterized by one $4$-tuple $(a, b, c, d)$ as above. One sees that $\#\max(I)$ equals the number of inequalities in \eqref{ideal-J2} which are strict.  This allows us to obtain $V_Q^{(1)}$. Let $J$ be another ideal  of $Q$ which is parameterized by the $4$-tuple $(a_1, b_1, c_1, d_1)$. Then $J\subseteq I$ if and only if $a_1\leq a$, $b_1\leq b$, $c_1\leq c$ and $d_1\leq d$. Moreover, in such a case, $\#\big(\max(I)\setminus J\big)$ equals the number of non-zero entries in the following sequence
$$
(a-a_1)(a-\min\{b, 4\}), \quad (b-b_1)(b-c), \quad (c-c_1)(c-\min\{6, d\}), \quad
(d-d_1)d.
$$
This allows us to obtain $A_Q$. Then aided by Theorem A, we find that $\caN_{[k]\times J^2([2]\times [3])}(x)$ are $\gamma$-positive for $k=5, 11$.
Indeed, $\caN_{[5]\times J^2([2]\times [3])}(x)$ has degree $10$ and $\gamma$-coefficients
$\gamma_0=1, 70, 745, 1850, 1025, 62$; while $\caN_{[11]\times J^2([2]\times [3])}(x)$ has degree $16$ and $\gamma$-coefficients
$$
\gamma_0=1, 160, 4900, 49280, 194810, 314720, 193760, 35840, 860.
$$
Therefore, in view of Lemma \ref{lemma-monic}(e), Conjecture B holds for $J^2([2]\times [3])$.
Moreover, we have checked that $\caN_{[k]\times J^2([2]\times [3])}(x)$ is log-concave for $1\leq k \leq 11$.\hfill\qed
\end{example}

\begin{rmk}\label{rmk-J3}
Similarly, we have computed  $\caN_{[k]\times J^3([2]\times [3])}(x)$ for $k$ up to $17$.
This polynomial is not palindromic when $k=9$ or $17$. For instance,
\begin{eqnarray*}
\caN_{[17]\times J^3([2]\times [3])}(x)=
1 + 459 x +  51867 x^2+ \cdots +20564929719672 x^{12}+ 26728302338920 x^{13} \\
 +26743721449352 x^{14}+ 20605116728504 x^{15} + \cdots + 55461 x^{25}+ 483 x^{26} + x^{27}.
\end{eqnarray*}
Thus Conjecture B holds for $J^3([2]\times [3])$ in view of Lemma \ref{lemma-monic}(f). Moreover, we have checked that this polynomial is log-concave for $1\leq k \leq 17$.
\end{rmk}

\begin{figure}[H]
\centering \scalebox{0.28}{\includegraphics{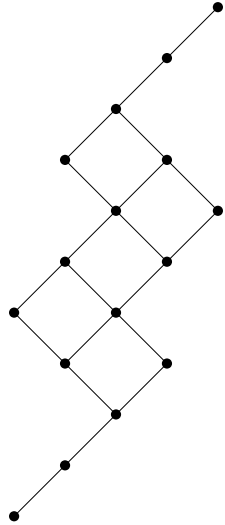}}
\caption{The Hasse diagram of $J^2([2]\times [3])$}
\label{fig-J2}
\end{figure}

To end up this section, let us present some examples suggesting that it is not easy to find  infinite family of $\gamma$-positive polynomials among $\caN_{[k]\times P}(x)$ for $P$ minuscule.

\begin{example}\label{example-many}
(a) The polynomial
$$
\caN_{[3]\times [3]\times [3]}(x)=1+27 x+162 x^2+350 x^3+310 x^4+114 x^5+15 x^6+x^7.
$$
It is not palindromic, and sits in the family $\caN_{[3]\times [n]\times [n]}(x)$.

(b) The polynomial  $\caN_{[3]\times [3]\times [5]}(x)$ equals
$$
1 + 45 x + 495 x^2 + 2155 x^3 + 4360 x^4 + 4360 x^5 + 2141 x^6 +
 505 x^7 + 49 x^8 + x^9.
$$
It is not palindromic and  sits in the family $\caN_{[3]\times [n]\times [n+2]}(x)$.

(c) The polynomial $\caN_{[4]\times [3]\times [4]}(x)$ equals
$$
1+48 x+576 x^2+2800 x^3+6525 x^4+7848 x^5+4957 x^6+1644 x^7+274 x^8+22 x^9+x^{10}.
$$
It is not palindromic, and sits in the family $\caN_{[4]\times [n]\times [n+1]}(x)$.

(d) The polynomial $\caN_{[4]\times [3]\times [6]}(x)$ looks like
$$
1 + 72 x + 1368 x^2 + \cdots + 103200 x^5 +
 134806 x^6 + 102912 x^7 + \cdots + 1510 x^{10} +
 86 x^{11} + x^{12}.
$$
It is not palindromic, and sits in the family $\caN_{[4]\times [n]\times [n+3]}(x)$.

(e) The polynomial $\caN_{[5]\times [3]\times [7]}(x)$ looks like
$$
1+105 x+3045 x^2+\cdots+4080285 x^7+4078275 x^8+\cdots+3692 x^{13}+137 x^{14}+x^{15}.
$$
It is not palindromic, and sits in the family $\caN_{[5]\times [n]\times [n+4]}(x)$.

(f) The polynomial $\caN_{[3]\times H_6}(x)$ equals
$$
1 + 63 x + 840 x^2 + 4088 x^3 + 8736 x^4 + 8736 x^5 + 4060 x^6 +
 862 x^7 + 69 x^8 + x^9.
$$
It is not palindromic, and sits in the family $\caN_{[3]\times H_n}$ for $n$ even.

(g) The polynomial $\caN_{[9]\times H_5}(x)$ looks like
$$
1 + 135 x + 4455 x^2  +\cdots+7209048 x^7 + 7206012 x^8+\cdots +4745 x^{13}+ 145 x^{14} + x^{15}.
$$
It is not palindromic, and sits in the family $\caN_{[2n-1]\times H_n}$.

All the polynomials above are log-concave. \hfill\qed
\end{example}

\section{The family $\caN_{[2] \times [n]\times [n+1]}(x)$}
This section aims to study the polynomials $\caN_{[2] \times [n]\times [n+1]}(x)$.
It follows from Lemma \ref{lemma-monic}(a) that $\caN_{[2]\times [n]\times [n+1]}(x)$ is monic and has degree $2n$.
By Theorem 6 of Proctor \cite{Pr84}, we have that
$$
\caN_{[2]\times [n]\times [n+1]}(1)=\prod_{\alpha\in [n]\times [n+1]} \frac{r(\alpha)+ 2}{r(\alpha)}=(2n+1)C_n C_{n+1},
$$
where $C_n=\frac{1}{n+1}{2n \choose n}$ is the $n$-th Catalan number.

Let us warm up with an example.
\begin{example}\label{exam-2-3}
We can enumerate all the ideals of $Q:=[2]\times [3]$ by the following $2$-tuples:
$$
(0, 0), (0, 1), (0, 2), (0, 3), (1, 1), (1, 2), (1, 3), (2, 2), (2,
3), (3, 3).
$$
If an ideal $I$ corresponds to the $2$-tuple $(a_1, a_2)$, we have that $\#\max (I)$ equals $\#\{a_i\mid a_i>0\}$. Thus one calculates that
$$
V_Q^{(1)}=[1, x, x, x, x, x^2, x^2, x, x^2, x]^T.
$$
Summing up these components leads us to that
$\caN_{ Q}(x)= 1+6x+3x^2$.

If another ideal $J$ corresponds to the $2$-tuple $(b_1, b_2)$, then $J\subseteq I$ if and only if $b_1\leq a_1$ and $b_2\leq a_2$. Then
$$
\#\big(\max(I)\setminus J\big)=\#\{a_i\mid a_i>b_i\},
$$
and one calculates that
\[
A_Q=
\begin{bmatrix}
1 & 0 & 0 & 0 &  0 & 0 & 0 & 0 &  0 & 0\\
x &  1 & 0 & 0 &  0 & 0 & 0 & 0 &  0 & 0 \\
x&x&1&0 &  0 & 0 & 0 & 0 &  0 & 0\\
x&x&x&1&0 & 0 & 0 & 0 &  0 & 0\\
x&x&0&0&1&0 & 0 & 0 &  0 & 0\\
x^2&x^2&x&0&x&1 & 0 & 0 &  0 & 0\\
x^2&x^2&x^2&x&x&x&1&0&0&0\\
x&x&x&0&x&x&0&1&0&0\\
x^2&x^2&x^2&x&x^2&x^2&x&x&1&0\\
x&x&x&x&x&x&x&x&x&1
\end{bmatrix}.
\]
Then for instance, by \eqref{N-I-k-K+1-matrix}, we have that
\begin{eqnarray*}
V_Q^{(4)}&=&A_Q^{3}V^{(1)}_Q\\
&=&[1, 4 x, 4 x + 6 x^2, 4 x + 12 x^2 + 4 x^3, 4 x + 6 x^2,
16 x^2 + 14 x^3, \\
& &16 x^2 + 38 x^3 + 11 x^4,
4 x + 18 x^2 + 22 x^3 + 6 x^4, 16 x^2 + 52 x^3 + 48 x^4 + 9 x^5,\\
& &4 x + 30 x^2 + 70 x^3 + 55 x^4 + 15 x^5 + x^6]^T.
\end{eqnarray*}
Summing up these components gives that
$$
\caN_{[4]\times Q}(x)= 1+24 x+120 x^2+200 x^3+120 x^4+24 x^5+x^6.
$$
There are many interlacing relations among $\caN_I^4(x)$. To mention a few, we have
\begin{eqnarray*}
\caN_{(0, 0)}^4(x)+ \caN_{(0, 1)}^4(x) &\preceq& \caN_{(0, 2)}^4(x),\\
\caN_{(0, 0)}^4(x)+ \caN_{(0, 1)}^4(x)+\caN_{(0, 2)}^4(x) &\preceq& \caN_{(0, 3)}^4(x),\\
\caN_{(1, 1)}^4(x)+ \caN_{(1, 2)}^4(x) &\preceq& \caN_{(1, 3)}^4(x),\\
\sum_{j=0}^{3}\caN_{(0, j)}^4(x) &\preceq&
\sum_{j=1}^{3}\caN_{(1, j)}^4(x),\\
\sum_{i=0}^{2}\sum_{j=i}^{3}\caN_{(i, j)}^4(x) &\preceq&
\caN_{(3, 3)}^4(x).
\end{eqnarray*}
\hfill\qed
\end{example}

Many other calculations lead us to the following.
\begin{conj}\label{conj-inter-mat} Fix $Q=[2]\times [n]$. Parameterize the ideals of $Q$ by the $2$-tuples $(a, b)$ such that $0\leq a\leq b\leq n$.  Fix any $k\geq 0$. Then the polynomials $\sum_{i=0}^{s}\sum_{l=i}^{n} \caN_{(i, l)}^k(x)$ and $\sum_{l=s+1}^{n}\caN_{(s+1, l)}^k(x)$ have common interleavers for $0\leq s\leq n-1$.
\end{conj}

If Conjecture \ref{conj-inter-mat} holds, setting $s=n-1$ there would give us that
$\sum_{i=0}^{n-1}\sum_{l=i}^{n} \caN_{(i, l)}^k(x)$ and $\caN_{(n, n)}^k(x)$ have a common interleaver.
Thus by 3.6 of \cite{CS},
$$
\caN_{[2]\times [n]\times [k]}(x)=\caN_{[k]\times Q}(x)=\sum_{i=0}^{n-1}\sum_{l=i}^{n} \caN_{(i, l)}^k(x)+\caN_{(n, n)}^k(x)
$$
would be real-rooted.

\begin{conj}\label{conj-real-rooted-2-n-k} The polynomial
$\caN_{[2]\times [n]\times [k]}(x)$ is real-rooted for any $n, k\in\bbZ_{>0}$.
\end{conj}
\begin{rmk}
The polynomial $\caN_{[3]\times [3]\times [3]}(x)$
(see Example \ref{example-many}(a)) is not real-rooted.
\end{rmk}

The polynomials $\caN_{[2]\times [n]\times [n+1]}(x)$ for the first few values of $n$ are computed as below. For convenience, here we only present the $\gamma$-coefficients $\gamma_0=1$, $\gamma_1$, $\dots$, $\gamma_n$ such that
$$
\caN_{[2]\times [n]\times [n+1]}(x)=\sum_{i=0}^{n} \gamma_i x^i (1+x)^{2n-2i}.
$$

\begin{center}
\begin{tabular}{l|r}
$n$ &  The $\gamma$-coefficients of $\caN_{[2]\times [n]\times [n+1]}(x)$ \\ \hline
 $1$ & $1$, $2$ \\
 $2$ & $1$, $8$, $2$ \\
 $3$ & $1$, $18$, $33$, $6$ \\
 $4$ & $1$, $32$, $150$ , $144$,  $12$ \\
 $5$ & $1$, $50$, $440$, $1040$, $580$, $40$\\
 $6$ & $1$, $72$, $1020$, $4480$, $6300$, $2400$, $100$\\
 $7$ & $1$, $98$, $2037$, $14350$, $37730$, $35700$, $9625$, $350$\\
 $8$ & $1$, $128$,  $3668$,  $37856$, $160020$, $282240$, $191100$, $39200$, $980$\\
 $9$ & $1$, $162$, $6120$, $ 87024$,  $539532$,
$1528632$, $1933344$,  $987840$,   $156996$, $3528$\\
$10$ & $1$, $200$, $9630$, $180480$, $1542660$, $6408864$, $13028400$, $12418560$, $4948020$, $635040$, $10584$
\end{tabular}
\end{center}

\begin{conj}\label{conj-2-n-n+1}
The polynomial $\caN_{[2]\times [n] \times [n+1]}(x)$ is palindromic and real-rooted (thus $\gamma$-positive) for every $n$.
\end{conj}

Finally, let us deduce a recursive formula for  $\caN_{[2] \times [n]\times [n+1]}(x)$ by Young tableau. Indeed, fix a Young tableau with two rows. Let the first row have length $n$, and the second row have length $m$, where $m\leq n$. We fill the boxes in the first row with numbers $0\leq a_1\leq a_2\leq \cdots\leq a_n\leq k$, and fill the boxes in the second row with numbers $0\leq b_1\leq \cdots \leq b_m\leq k$. Require that $a_i\leq b_i$ for $1\leq i\leq m$. For instance, when $n=2$, $m=1$ and $k=1$, there are five such Young tableaux:
\[
\young(00,0) \qquad\young(01,0)\qquad \young(00,1)  \qquad \young(01,1) \qquad \young(11,1)
\]
Given such a Young tableau, we associate the monomial
$$
x^{\#\{a_i\mid a_i>0\}+\# \{b_j \mid b_j>a_j \}}
$$
to it. For instance, the monomials associated to the five Young tableaux above are
$1, x, x, x^2, x$ respectively. We denote by $f_{n, m}^k$ the sum of all the monomials associated to the Young tableaux defined above with row lengths $n$, $m$ and filling numbers no bigger than $k$. For example,
$f_{2, 1}^1=1+3x+x^2$.

\begin{prop}\label{lemma-recursive-young} We have the recursive formula
\begin{equation}
f_{n, n}^{k+1}=f_{n, n}^k+x\sum_{s=1}^{n} f_{n, n-s}^k+\sum_{l=1}^n(xf_{n-l, n-l}^{k}+x^2\sum_{s=1}^{n-l}f_{n-l, n-l-s}^k).
\end{equation}
\end{prop}
\begin{proof}
Assume that in the first row we fill precisely the last $l$ boxes with $k+1$, while in the second row we fill precisely the last $s+l$ boxes with $k+1$. We proceed according to the following two cases.

\noindent \emph{Case 1:} $l =0$. Then $0\leq s\leq n$.

\noindent \emph{Subcase 1a:} $s=0$. Then one sees easily that the monomials associated to these Young tableaux sum up to $f_{n, n}^k$.

\noindent \emph{Subcase 1b:}  $1 \leq s \leq n$. Then we have that
\begin{eqnarray*}
\#\{a_i\mid a_i>0\}+\# \{b_j \mid b_j>a_j \}
&=&\#\{a_i\mid a_i>0 \ \text{for} \ i\in[1,n]\}\\
&+&\# \{b_j \mid b_j>a_j \ \text{for} \ j\in[1,n-s]\}+1.
\end{eqnarray*}
The monomials associated to these Young tableaux sum up to $x\sum_{s=1}^{n} f_{n, n-s}^k$.

\noindent \emph{Case 2:}  $1 \leq l \leq n$. Then $0\leq s\leq n-l$.

\noindent \emph{Subcase 2a:} If $s=0$, we obtain that
\begin{eqnarray*}
\#\{a_i\mid a_i>0\}+\# \{b_j \mid b_j>a_j \}
&=&\#\{a_i\mid a_i>0 \ \text{for} \ i\in[1,n-l]\}\\
&+&\# \{b_j \mid b_j>a_j \ \text{for} \ j\in[1,n-l]\}+1.
\end{eqnarray*}
The monomials associated to these Young tableaux sum up to $x\sum_{l=1}^{n} f_{n-l, n-l}^k$.

\noindent \emph{Subcase 2b:} $1\leq s\leq n-l$. Then we have that
\begin{eqnarray*}
\#\{a_i\mid a_i>0\}+\# \{b_j \mid b_j>a_j \}
&=&\#\{a_i\mid a_i>0 \ \text{for} \ i\in[1,n-l]\}\\
&+&\#\{b_j \mid b_j>a_j \ \text{for} \ j\in[1,n-l-s]\}+2.
\end{eqnarray*}
The monomials associated to these Young tableaux sum up to $x^2\sum_{l=1}^{n}\sum_{s=1}^{n-l} f_{n-l, n-l-s}^k$.

Adding up the four terms above gives the desired formula.
\end{proof}

Note that $f_{n, n}^{n+1}$ is the $\caN$-polynomial for $[2]\times [n] \times [n+1]$. Thus Proposition \ref{lemma-recursive-young}  may be useful for investigating the conjectures of this section.

\medskip

\centerline{\scshape Acknowledgements}
Sincere thanks go to Professor Athanasiadis for telling us the proof of Proposition \ref{prop-gamma-H-n}, to Professor Proctor for carefully explaining to us that every connected minuscule poset is Peck, and to Professor Stanley for sharing his immense knowledge with us.

\end{document}